\documentclass[freqno,12pt]{amsart}%[draft]
\usepackage{amscd,amsfonts,amsmath,amssymb,amstext,amsthm}

\usepackage[utf8]{inputenc}
\usepackage[T2A]{fontenc}
\usepackage[russian,english,]{babel}
\usepackage{cite}
\usepackage[unicode]{hyperref}
\usepackage{color}
\usepackage{xcolor}
\usepackage{comment}
\usepackage[normalem]{ulem}

\makeatletter
\def\@settitle{\begin{center}%
    \baselineskip14\p@\relax
    \bfseries
    \MakeUppercase{\@title}
  \end{center}
}
\makeatother

\newtheorem{theorem}{Theorem}%[section]
\newtheorem{lemma}{\it Lemma}[section]
\newtheorem{proposition}{Proposition}[section]
\theoremstyle{remark}

\newtheorem{remark}{Remark}[section]
\newtheorem{example}{\it Example}[section]
\theoremstyle{definition}
\newtheorem{assert}{Assertion}
\newtheorem{definition}{\it Definiton}[section]
\numberwithin{equation}{section}

\def\C{{\mathbb C}}

\def\Cn*{{\C^n}^*}
\def\CN*{{\C^N}^*}

\def\P{{\mathbb P}}

\def\R{{\mathbb R}}
\def\Z{{\mathbb Z}}

\def\vol{{\rm vol}}
\def\E{\:{\rm e}}

\begin{document}
\title{Around the ``Fundamental Theorem of Algebra''
}
{
\author{B. Kazarnovskii
}
\address {\noindent
Moscow Institute of Physics and
Technology (National Research University),
Higher School of Contemporary Mathematics
\newline
{\it kazbori@gmail.com}.}
%}
%\thanks {%\noindent
%%\support {\noindent
%%\newline
%This research was carried out at the Higher School of Contemporary Mathematics of Moscow
%Institute of Physics and Technology,
%with the support of the Ministry of Science and Higher Education
%of the Russian Federation,
%project no. SMG - 2024 -0048.
%}
\thanks {%\noindent
The research is supported by the MSHE "Priority 2030" strategic academic leadership program}

\begin{abstract}
The Fundamental Theorem of Algebra (FTA) asserts that every complex polynomial has as many complex roots, counted with multiplicities, as its degree.
A probabilistic analogue of this theorem for real roots of real polynomials—commonly referred to as the Kac theorem—was introduced in 1938 by J. Littlewood and A. Offord.
In this paper, we present the Kac theorem and prove two more theorems that can be interpreted as analogues of the FTA:
a version of FTA for real Laurent polynomials, and another version for exponential sums.
In these two cases, we also provide formulations of multidimensional analogues of corresponding FTA.
 While these results are not new, they may appear unexpected and are therefore worth highlighting.
\end{abstract}
\maketitle
\tableofcontents
\section{Kac Theorem}\label{Kac}
A version of the Fundamental Theorem of Algebra (FTA) concerning the real roots of a real polynomial,
based on the concept of the probability that a root is real (sometimes referred to as the Kac theorem),
was established in 1938 by J.~Littlewood and A.~Offord; see~\cite{Littl1,Littl2,Littl3,KA}.
The origin of the Kac theorem lies in the following simple observation:
for any integer~$m$, the probability that a root of a real polynomial of degree~$m$ is real is nonzero.
For instance, if $m=1$, this probability equals~$1$.

\begin{theorem}\label{thmKac}
Let the coefficients of a random real polynomial of degree~$m$ in one variable be independent, normally distributed random variables with zero mean and unit variance.
Then, as $m\to\infty$, the expected number of real roots is asymptotically $\frac{2}{\pi}\log m$.
Equivalently, $\mathcal P(m)\asymp \frac{2}{\pi}\frac{\log m}{m}$,
where $\mathcal P(m)$ denotes the probability that a root of a random real polynomial of degree~$m$ is real.
\end{theorem}

\begin{remark}
For further details on the distribution of the number of real roots of random real polynomials,
see~\cite{EK} and the references therein.
\end{remark}

\begin{remark}
If the expected value of the $k$th coefficient of a random polynomial is assumed to be $C_m^k$,
then in Kac’s theorem the term $\sqrt m$ appears in place of $\log m$.
\end{remark}

\section{FTA for Laurent polynomials}\label{Laur}
\subsection{Formulation of the FTA}\label{Laur1}
Recall that Laurent polynomial is a function on $\C\setminus 0$ of the form
$$\sum\nolimits_{m\in\Lambda\subset\Z,\: a_m\in\C}a_m z^m,$$
where the finite set $\Lambda$ in $\Z$  is called the \emph{spectrum of Laurent polynomial.}
\begin{definition}\label{dfLaur}

(1) A Laurent polynomial $P$ is called \emph{real} if its values on the circle $S$ of radius $1$ centered at the origin are real.

(2)
Any root lying in $S$ is called a \emph{real root of the Laurent polynomial}.

(3)
By definition, the \emph{degree of the real Laurent polynomial} is $\max_{k\in\Lambda} |k|$
\end{definition}
\begin{lemma}\label{lmL1}
{\rm(1)} Laurent polynomial $P$ is real if and only if $$\forall k\in\Z\colon a_k=\overline{a_{-k}}.$$
In particular, the spectrum of a real Laurent polynomial is centrally symmetric.

{\rm(2)} The set of roots of a real Laurent polynomial is invariant under the inversion mapping
$z\mapsto \bar z^{-1}$ with center $S$.

{\rm(3)} The restriction mapping of real Laurent polynomials from
$\C\setminus0$ to $S$ is
%defines
an isomorphism of the space of real Laurent polynomials of degree $\leq m$ to the space of trigonometric polynomials  with real coefficients of the form
\begin{equation}\label{tr}
 a_0+\sum_{1\leq k\leq m} \alpha_k\cos(k\theta)+\beta_k\sin(k\theta)
\end{equation}
\end{lemma}
\begin{definition}\label{dfTrig}(Randomness for real Laurent polynomials.)
Let ${\rm Trig}_m$ be the space of trigonometric polynomials of the form (\ref{tr}).
We consider ${\rm Trig}_m$ as a subspace in $L_2(S)$,
i.e. in the space of functions on $S$ with scalar product $(f,g)=\frac{1}{2\pi}\int\nolimits_Sf g$.
We assume the points in ${\rm Trig}_m$ and, accordingly, real Laurent polynomials (using Lemma \ref{lmL1} (3))
as normally distributed with respect to the standard
normal distribution associated with the metric taken from $L_2(S)$.
\end{definition}
\begin{theorem}\label{thmL1}
Let $f_m$ be a random real Laurent polynomial of degree $m$.
Then

{\rm(a)}
the expectation of the number of real roots of $f_m$ is
 $2\sqrt{\frac{m(m+1)}{3}}$

 {\rm(b)}
a probability $\mathcal P(f)$ of a root of random Laurent polynomial
 $f$ of degree $m$ being real equals
 $\sqrt{\frac{m+1}{3m}}$

 {\rm(c)}
 $\lim_{m\to\infty}\mathcal P(f_m)=1/\sqrt 3$
\end{theorem}
\begin{remark}
These statements are equivalent to computation the average number of zeros
of trigonometric polynomials of fixed degree on a circle;
see \cite{ADG}.
\end{remark}
When replacing the spectrum
$\Lambda=\{-m,\ldots,0,\ldots,m\}$ by any
centrally symmetric finite set $\Lambda$ in $\Z$,
then %the statements
(a) and (b) turn into the
statements (a) and (b) of the following statement.

\begin{theorem}\label{thmL2}
Let $f_\Lambda$ be a random real Laurent polynomial with the spectrum $\Lambda$.
Then

{\rm(a)}
the mean number of real roots of $f_\Lambda$ is
$$
2\cdot \sqrt{\frac{1}{\#\Lambda}\sum\nolimits_{\lambda\in\Lambda}\lambda^2}
$$
where $\#\Lambda$ is a number of elements in $\Lambda$

 {\rm(b)}
the probability of a root of $f_\Lambda$ being real is
$$
  \mathcal P(\Lambda)=\frac{1}{\deg(f_\Lambda)}
\sqrt{\frac{1}{\#\Lambda}\sum\nolimits_{\lambda\in\Lambda}\lambda^2}
$$

 {\rm(c)}
when replacing the spectrum
$\Lambda$ by the spectrum $k\Lambda$,
the probability of a root being real does not change
(obviously follows from the fact that the circle $S$ is a subgroup of the group $\C\setminus0$)
\end{theorem}
\begin{example}
If $k>0$ and $\Lambda=\{-k,k\}$
then  $\mathcal P(\Lambda)=\frac{1}{k}\sqrt{\frac{1}{2}2k^2}=1$.
\end{example}
\subsection{Proof of Theorem \ref{thmL2}}\label{Laur2}
To prove Theorem 3 we use the simplest version of the Crofton-type formula.
\begin{proposition}\label{prCr}
Let

$\bullet$ $\mu$  -- Gaussian measure in $\R^n$ and $S^{n-1}$ -- unit sphere in $\R^n$
centered at zero

$\bullet$ $Z(\xi)=\{x\in\R^n\colon (x,\xi)=0\}$, where $0\ne\xi\in\R^n$

$\bullet$ $K$ -- compact curve in $S^{n-1}$

Then
 $$\int_{\R^n}\#(K\cap Z(\xi))\:d\mu(\xi)=(1/\pi)\:{\rm length}(K)$$
 \end{proposition}
Let  ${\rm Trig}(\Lambda)$ be the space of trigonometric polynomials with the centrally symmetric spectrum $\Lambda$
(see Definition \ref{dfTrig}),
$n=\dim {\rm Trig}(\Lambda)$ (i.e. $n=\#\Lambda$),
$f_1,\ldots,f_n$ be an orthonormal basis in ${\rm Trig}(\Lambda)$.
\begin{lemma}\label{lmL2}
{\rm(1)}
Let $F_\theta=f_1(\theta)f_1+\ldots+f_n(\theta)f_n\in {\rm Trig}(\Lambda)$.
Then
 $$\forall \theta\in S,\,f\in{\rm Trig}(\Lambda)\colon\,f(\theta)=(f,F_\theta)$$

{\rm(2)}
The mapping
$\kappa\colon\theta\mapsto(1/\sqrt n) F_\theta$
does not depend on the choice of $\{f_i\}$.

{\rm(3)}
$\kappa(S)$ belongs to
a sphere of radius $1$ centered at zero
in ${\rm Trig}(\Lambda)$
\end{lemma}
\begin{proof}

(1):\ $(f_k,F_\theta)=(f_k,\sum_i f_i(\theta)f_i)=(f_k,f_k(\theta)f_k)=f_k(\theta)$.

(2):\  It follows from the statement (1).

(3):\ The action of $S$ on the space ${\rm Trig}(\Lambda)$ preserves the metric.
Hence the function $|F_\theta|^2$ on $S$ is constant.
Integrating over $S$,
we obtain the (3).
\end{proof}
For $0\ne\lambda\in\Lambda$, we consider in %the space
${\rm Trig}(\Lambda)$ the vectors
$
\phi_\lambda(\theta)=\sqrt2\cos(\lambda\theta)$, $\,\psi_\lambda(\theta)=\sqrt2\sin(k\theta)$.
If $0\in\Lambda$ then we set $\phi_0(\theta)=\sqrt{2\pi}$.
\emph{The elements $\{\phi_\lambda,\psi_\lambda\}$ form an orthonormal basis in ${\rm Trig}(\Lambda)$.}
From this we obtain
$$
{\rm length}(\kappa(S))=
\int_S\sqrt{\frac{1}{\#\Lambda}\sum\nolimits_{\lambda\in\Lambda}\lambda^2}d\theta=
2\pi\sqrt{\frac{1}{\#\Lambda}\sum\nolimits_{\lambda\in\Lambda}\lambda^2}
$$
Applying proposition \ref{prCr} to the spherical curve $\kappa(S)$, we obtain statement (a) of the Theorem \ref{thmL2}.
\subsection{Multidimensional analog of Theorem \ref{thmL2}}\label{Laur3}
We now %temporarily
turn to the case of polynomials in several variables.
Namely we formulate a multidimensional analogue of Theorem \ref{thmL2} from \cite{K22}.

Let $\Lambda$ be a finite centrally symmetric subset of $\Z^n$.
We recall that a Laurent polynomial in $n$ variables with spectrum $\Lambda$ is a function
$\sum_{\lambda\in\Lambda}a_\lambda z^\lambda$ on the complex torus $(\C\setminus0)^n$.
A Laurent polynomial that takes real values on the compact subtorus
$$
T^n={z\in(\C\setminus{0})^n : z=(\E^{i\theta_1},\ldots,\E^{i\theta_n})}
$$
of $(\C\setminus{0})^n$ is called a \emph{real Laurent polynomial}.
If a Laurent polynomial equals zero at a point $z\in T^n$,
then $z$ is called the \emph{real zero} of the Laurent polynomial.

It turns out that the expected number of real zeros of a system of random real Laurent polynomials in $n$ variables with a fixed spectrum~$\Lambda$ equals the volume of a certain $n$-dimensional ellipsoid multiplied by $n!$.
Namely, we define \emph{the Newton ellipsoid} ${\rm Ell}(\Lambda)$ as an ellipsoid with a support function
$$
h(x)=\sqrt{\frac{1}{\#\Lambda}\sum\nolimits_{\lambda\in\Lambda}\lambda^2(x)},
$$
where each $\lambda$ is regarded as a linear functional on $\R^n$.
For $n=1$, the Newton ellipsoid is the segment $[-h(1),h(1)]$,
and Theorem~\ref{thmL2}(a) is recovered.

Applying Koushnirenko theorem for the number of solutions of a polynomial system,
we obtain the following formula for the probability $\mathcal P(\Lambda)$ that a root is real:
\begin{equation}\label{eqTrig}
  \mathcal P(\Lambda)=\frac{\vol\left({\rm Ell}(\Lambda)\right)}{\vol\left({\rm conv}(\Lambda)\right)}
\end{equation}
where ${\rm conv}(\Lambda)$ denotes the convex hull of the spectrum~$\Lambda$.
In the case $n=1$, this formula reproduces Theorem~\ref{thmL2}(b).
In~\cite{K22}, the limiting value of the probability for an increasing sequence of spectra~$\Lambda$ is computed,
yielding a multidimensional analogue of Theorem~\ref{thmL1}(c).

The probability of a root being real can also be defined for systems of Laurent polynomials with distinct spectra.
In this case, the volumes in the numerator and denominator of the fraction (\ref{eqTrig})
are replaced by the mixed volumes of the corresponding ellipsoids and polyhedra.
Mixed volumes of ellipsoids in connection with the problem of zeros of random systems of equations
first appeared in~\cite{ZK}.
\begin{remark}
By viewing Laurent polynomials as functions associated with representations of a torus,
one can describe an analogous phenomenon for representations of arbitrary reductive linear groups;
see~\cite{K25}.
\end{remark}
\section{FTA for exponential sums}\label{exp}
\subsection{Exponential sums of one variable}\label{exp1}
An exponential sum is a function in $\C$ of the form
$$
  f(z)=\sum\nolimits_{\lambda\in\Lambda\subset\C,\:c_\lambda\in\C}c_\lambda\E^{\bar\lambda z}
$$
The convex hull $\Delta$ of the finite set $\Lambda$ is called the \emph{Newton polygon} of the exponential sum.
The zero set of an exponential sum is always infinite.
For instance, the zeros of the function $\E^{2\pi i z} - 1$ are precisely the integers $\Z$.
The following result can be viewed as an analogue of the Fundamental Theorem of Algebra for exponential sums.
\begin{theorem}\label{thmexp}
Let $N(f,r)$ be the number of zeros of exponential sum $f$ in the disk $B_r$ of radius $r$ centered at the origin.
 Then
\begin{equation}\label{expMain}
 N(f,r)=\frac{r}{2\pi}\:l(f)+ O(1),
\end{equation}
where $l(f)$ is the semiperimeter of the Newton polygon $\Delta$ of $f$.
\end{theorem}
\begin{remark}
If the convex polygon $\Delta$ degenerates to a line segment, then its perimeter approaches twice the length of that segment.
Hence, when the Newton polygon of $f$ is a segment, the quantity $l(f)$ in~\eqref{expMain} equals the length of this segment.
\end{remark}
\begin{remark}
After the change of variables $z\mapsto\E^z$, the Laurent polynomial becomes an exponential sum.
Therefore, Theorem 4 implies FTA.
\end{remark}
\emph{Proof of Theorem \ref{exp}.}
Let $r_1,\ldots,r_N$ be the rays generated by the outward normals to the sides
$\Delta_1,\ldots,\Delta_N$ of the Newton polygon $\Delta$ associated with the exponential sum $f$.
For each $j$, let $U_{j,R}$ denote the set of points in~$\C$ whose distance from~$r_j$ does not exceed~$R$.
If $R$ is sufficiently large, then all zeros of~$f$ lie in
$\bigcup_{1\le j\le N} U_{j,R}$.
Denote by $N_j(f,r)$ the number of zeros of~$f$ contained in~$U_{j,R}$.
Using the argument principle, one readily obtains
\begin{equation}\label{side}
  N_j(f,r) = \frac{r}{2\pi}\, l_j(f) + o(r),
\end{equation}
where $l_j(f)$ is half the length of the side~$\Delta_j$.
Summing over~$j$, we obtain the asymptotic relation~\eqref{expMain} for the leading term.
It remains to show that the $o(r)$ term in~\eqref{side} can in fact be replaced by an $O(1)$ term.
This follows from the next statement.
\begin{lemma}
For any compact set $A \subset \C$ there exists a constant $C(A) > 0$ such that,
for all $z \in \C$, the number of zeros of the exponential sum~$f$ in the set $z + A$
does not exceed~$C(A)$.
\end{lemma}
\begin{proof}
Consider the space of exponential sums with a fixed spectrum~$\Lambda$,
and let $\P$ denote its projectivization.
We identify points of~$\P$ with exponential sums defined up to a nonzero scalar factor.
Under this identification, the shift action of~$\C$ on exponential sums,
given by $w \colon f(z) \mapsto f(z + w)$,
induces a projective representation
$$
  \rho \colon \C \to \mathrm{Aut}(\P)
$$
of the additive group~$\C$ in the projective space~$\P$.

Suppose, to the contrary, that the constant $C(A)$ does not exist.
Then there exists a sequence of points $w_1, w_2, \ldots$ in~$\C$
such that the number of zeros of the function $f(w_k + z)$ in the set~$A$
tends to infinity as $k \to \infty$.
We can choose this sequence so that the corresponding sequence $\rho(w_k)(f)$
converges in~$\P$ to a limit~$f_\infty$.
However, this would imply that the exponential sum~$f_\infty$ has infinitely many zeros in~$A$,
which is impossible.
\end{proof}
\subsection{Exponential sums of several variables}\label{exp2}
Now let us turn to the case of exponential sums in several variables
and briefly describe the multidimensional analogue of Theorem~\ref{thmexp};
see, for example,~\cite{exp}.
Recall that an exponential sum in~$n$ variables is a function on~$\C^n$ of the form
$$
f(z)=\sum_{\lambda\in\Lambda,\:c_\lambda\in \C} c_\lambda \E^{\lambda(z)},
$$
where $\Lambda$ is a finite subset of the dual space~$\C^{n*}$ consisting of linear functionals on~$\C^n$.
In what follows, we refer to~$\Lambda$ as the \emph{spectrum} of the exponential sum~$f$.

For a tuple of $n$ exponential sums $F = (f_1, \ldots, f_n)$ with a common spectrum~$\Lambda$,
let $N(F, r)$ denote the number of common isolated zeros of the exponential sums
$f_1, \ldots, f_n$ inside the ball~$B_r$ of radius~$r$ centered at the origin.
The following statement is a multidimensional analogue of Theorem~\ref{thmexp}.
\begin{assert}
For almost all tuples of exponential sums $F = (f_1, \ldots, f_n)$ with spectrum~$\Lambda$, one has
$$
N(F,r)={\rm pvol}\left({\rm conv}(\Lambda)\right)\frac{r^n}{(2\pi)^n}+O(r^{n-1}),
$$
where $\mathrm{pvol}(\Delta)$ denotes the \emph{pseudovolume} of the convex polyhedron
$\Delta \subset \C^{n*}$, as defined below in Definition~\ref{dfpvol}.
\end{assert}
\begin{definition}\label{dfpvol}
Let $\Delta$ be a convex polyhedron in the space~$\C^{n*}$.
Define
$$
{\rm pvol}(\Delta)=
%\sigma_n
\sum_{\Gamma\subset\Delta,\:\dim(\Gamma)=n}c(\Gamma)\:A(\Gamma)\:\vol_n(\Gamma),
$$
where the summation is taken over all $n$-dimensional faces~$\Gamma$ of the polyhedron~$\Delta$,
and the quantities $\vol_n(\Gamma)$, $A(\Gamma)$, and $c(\Gamma)$ are defined as follows.
\begin{enumerate} %[(i)]
  \item
  $\vol_n(\Gamma)$ is the $n$-dimensional volume of the face~$\Gamma$.

  \item $A(\Gamma)$ is the \emph{exterior angle} of~$\Gamma$,
  i.e., the angle of the dual cone~$K_\Gamma$ corresponding to the face~$\Gamma$
  (the full $n$-dimensional angle is normalized to~$1$).
  Recall that $K_\Gamma$ consists of those points $w \in \C^n$
  for which the linear functional $f(z) = \Re\, z(w)$, considered as a function on~$\Delta$,
  attains its maximum simultaneously at all points of the face~$\Gamma$.

  \item $c(\Gamma) = \cos(T_\Gamma^\perp, \sqrt{-1}\,T_\Gamma)$
  is the cosine of the angle between the subspaces $T_\Gamma^\perp$ and $\sqrt{-1}\,T_\Gamma$,
  where $T_\Gamma \subset \C^{n*}$ is the tangent space to the face~$\Gamma$,
  and $T_\Gamma^\perp \subset \C^{n*}$ is its orthogonal complement.
  Recall that the cosine of the angle between two $n$-dimensional subspaces is,
  by definition, the area distortion coefficient of~$\vol_n$
  under the orthogonal projection of one subspace onto the other.
\end{enumerate}
\end{definition}
\begin{example}\label{ex3}

\begin{enumerate}
  \item If $\Delta \subset {\rm re}(\C^{n*})$, then
    $\mathrm{pvol}(\Delta) = \vol_n(\Delta)$.

  \item If $n = 1$, then $\mathrm{pvol}(\Delta)$ equals the semiperimeter of the polygon~$\Delta$,
  in accordance with Theorem~\ref{exp}.
\end{enumerate}
\end{example}

\end{document}